\documentclass[letterpaper, 10 pt, conference]{ieeeconf}  

\IEEEoverridecommandlockouts                              

\usepackage{graphics} 
\usepackage{epsfig} 
\usepackage{mathptmx} 
\usepackage{times} 
\usepackage{amsmath} 
\usepackage{amssymb}  
\usepackage{dsfont}
\usepackage{booktabs}
\usepackage{mathtools}

\usepackage{nicefrac}       
\usepackage{microtype}      
\usepackage{tabto}
\usepackage{cite}

\usepackage{enumerate}
\usepackage{caption}
\usepackage{subcaption}

\usepackage{relsize}
\newcommand{\tra}{{\mathsmaller T}}
\newcommand{\Bt}{{\mathsmaller B_t}}
\newcommand{\Btm}{{\mathsmaller B_{t-1}}}

\usepackage[notocbasic]{nomencl}
\makenomenclature
\usepackage{etoolbox}

\usepackage{algorithm}
\usepackage[noend]{algorithmic}

\usepackage{hyperref}       
\usepackage{url}            
\hypersetup{ 		
    colorlinks=true,        
    linkcolor=NavyBlue,
    urlcolor=NavyBlue,
    citecolor=NavyBlue     		
}
\usepackage[usenames,dvipsnames]{color}

\usepackage{xpatch}

\newtheorem{thm}{Theorem}

\title{\LARGE \bf
Self-Tuning Network Control Architectures
}

\author{ Tyler Summers \quad Karthik Ganapathy \quad Iman Shames \quad Mathias Hudoba de Badyn
\thanks{T. Summers and K. Ganapathy are with the Control, Optimization, and Networks Lab, University of Texas at Dallas (email: \{tyler.summers\}@utdallas.edu). This material is based on work supported by the United States Air Force Office of Scientific Research under award number FA2386-19-1-4073 and by the National Science Foundation under award number ECCS-2047040. I. Shames is the the Australian National University. M. Hudoba de Badyn is with the Automatic Control Laboratory at ETH Z\"urich.}
\\
\fbox{
\parbox{6.25in}{Accepted for publication in 61st Conference on Decision and Control pp.~5876--5881 \textcopyright 2022 IEEE.}
}
\vspace{-0.785cm}
}

\begin{document}

\maketitle
\thispagestyle{empty}
\pagestyle{empty}

\begin{abstract}
We formulate a general mathematical framework for self-tuning network control architecture design. This problem involves \emph{jointly} adapting the locations of active sensors and actuators in the network and the feedback control policy to all available information about the time-varying network state and dynamics to optimize a performance criterion. We propose a general solution structure analogous to the classical self-tuning regulator from adaptive control. We show that a special case with full-state feedback can be solved in principle with dynamic programming, and in the linear quadratic setting the optimal cost functions and policies are piecewise quadratic and piecewise linear, respectively. For large networks where exhaustive architecture search is prohibitive, we describe a greedy heuristic for joint architecture-policy design. We demonstrate in numerical experiments that self-tuning architectures can provide dramatically improved performance over fixed architectures. Our general formulation provides an extremely rich and challenging problem space with opportunities to apply a wide variety of approximation methods from stochastic control, system identification, reinforcement learning, and static architecture design.
\end{abstract}

\section{Introduction} \label{sec:intro} 
Emerging complex dynamical networks present tremendous challenges and opportunities to fundamentally reimagine their control architectures and algorithms. In large-scale networks, the structure of the control architecture -- the locations of sensors, actuators, and their communication pattern -- is crucial to performance and robustness properties and an important design consideration. Much of control theory operates with fixed control architectures, with the design focused almost entirely on the control policy rather than the architecture. There is a long history in adaptive control and reinforcement learning of adapting policy parameters online to measured data and/or identified models, but these ideas have never been applied to the architecture itself in a feedback loop with measured data. Here we propose \emph{self-tuning network control architectures} that jointly adapt the policy \emph{and} architecture online to measured data, in analogy to the classical self-tuning regulator in adaptive control \cite{astrom2013adaptive}.

Such self-tuning architectures are compelling for emerging large infrastructure networks and complex, high-dimensional networks with time-varying phenomena. This includes, for example, power grids with massive penetration of inverter-based resources, mixed-mode transportation networks, epidemic/information spread in social networks, and economic activity in large economies. New sensing and actuation technologies are being rapidly integrated and offer an increasingly large number of points from which to estimate and control complex dynamic phenomena. However, resource and budget constraints may limit the number of sensors and actuators that are active at a given time. When the network state and/or dynamics evolve to conditions that limit the effectiveness of a fixed set of sensors and actuators, self-tuning architectures can provide significantly improved performance and robustness.


There is a vast and rapidly growing literature in adaptive control \cite{astrom2013adaptive} and reinforcement learning (RL) \cite{sutton2018reinforcement,bertsekas2019reinforcement} that focuses on adapting policy parameters based on measured data. The self-tuning regulator \cite{astrom2013adaptive} is a prototypical adaptive control approach where model parameters related to the system dynamics are estimated online from data and then used to adapt the parameters of a feedback policy through a control design procedure. This approach can flexibly accommodate many combinations of parameter estimation and control design methodologies. There are several other approaches in adaptive control and RL of indirect and model-free flavors that adapt policy parameters through estimation of other quantities, such as value functions and policy gradients. However, most work in adaptive control and RL focuses on individual systems, and work in a network context utilizes architectures with fixed locations for sensors and actuators.

Designing network control architecture to optimize controllability and observability metrics has received considerable attention in recent years \cite{olshevsky2014minimal,pasqualetti2014controllability,tzoumas2015minimal,summers2016submodularity,ruths2014control, de2020mathcal,foight2020performance}. However, the architecture design is largely treated as a single static design problem. Some works have studied time-varying actuator scheduling \cite{zhao2016scheduling,nozari2017time,siami2020deterministic,olshevsky2020relaxation,siami2020separation}, but the architecture, while time-varying, remains in open-loop and does not adapt to changing network state or dynamics. Very recent work has considered selecting actuators for uncertain systems based on data measured from the system in limited settings with linear systems and specific controllability metrics \cite{fotiadis2021learning,ye2022online}.

Here we aim to bring together in a broad way network control architecture design with feedback, adaptive control, and RL. Our main contributions are as follows:
\begin{enumerate}
    \item We formulate a general mathematical framework for self-tuning network control architecture design and propose a general solution structure analogous to the classical self-tuning regulator from adaptive control.
    \item For a special case with full-state feedback and known dynamics, we show that dynamic programming can solve the problem in principle, which couples feedback policy design with a search over architecture combinations. 
    \item In the linear quadratic setting, we show that the optimal cost functions and policies are piecewise quadratic and affine, respectively. We also propose a computationally tractable greedy heuristic for self-tuning LQR architectures.
    \item We demonstrate in numerical experiments that self-tuning architectures can provide dramatically improved performance over fixed architectures.
\end{enumerate}
Our general formulation provides an extremely rich and challenging problem space with opportunities to apply a wide variety of approximation methods from stochastic control, system identification, reinforcement learning, and static architecture design.

\section{Self-Tuning Network Control Architectures} \label{sec:prelim}
We first formulate a general mathematical framework for self-tuning network control architectures. Consider a dynamical network with underlying graph $G = (V, E(t))$, where $V=\{1,...,n\}$ is a set of nodes and $E(t) \subseteq V \times V$ is a time-varying set of edges connecting nodes over a discrete time horizon $t\in[0,\dots,\mathcal{T}]$. We associate a state variable $x_i(t) \in \mathcal{X}_i$ with node $i \in V$. The network state is $x(t) = [x_1(t), x_2(t), ..., x_n(t)] \in \mathcal{X} = \Pi_{i=1}^n \mathcal{X}_i$. The edges represent dynamical interactions between nodal states.

\begin{figure}[ht]
    \centering
    \includegraphics[width=\linewidth]{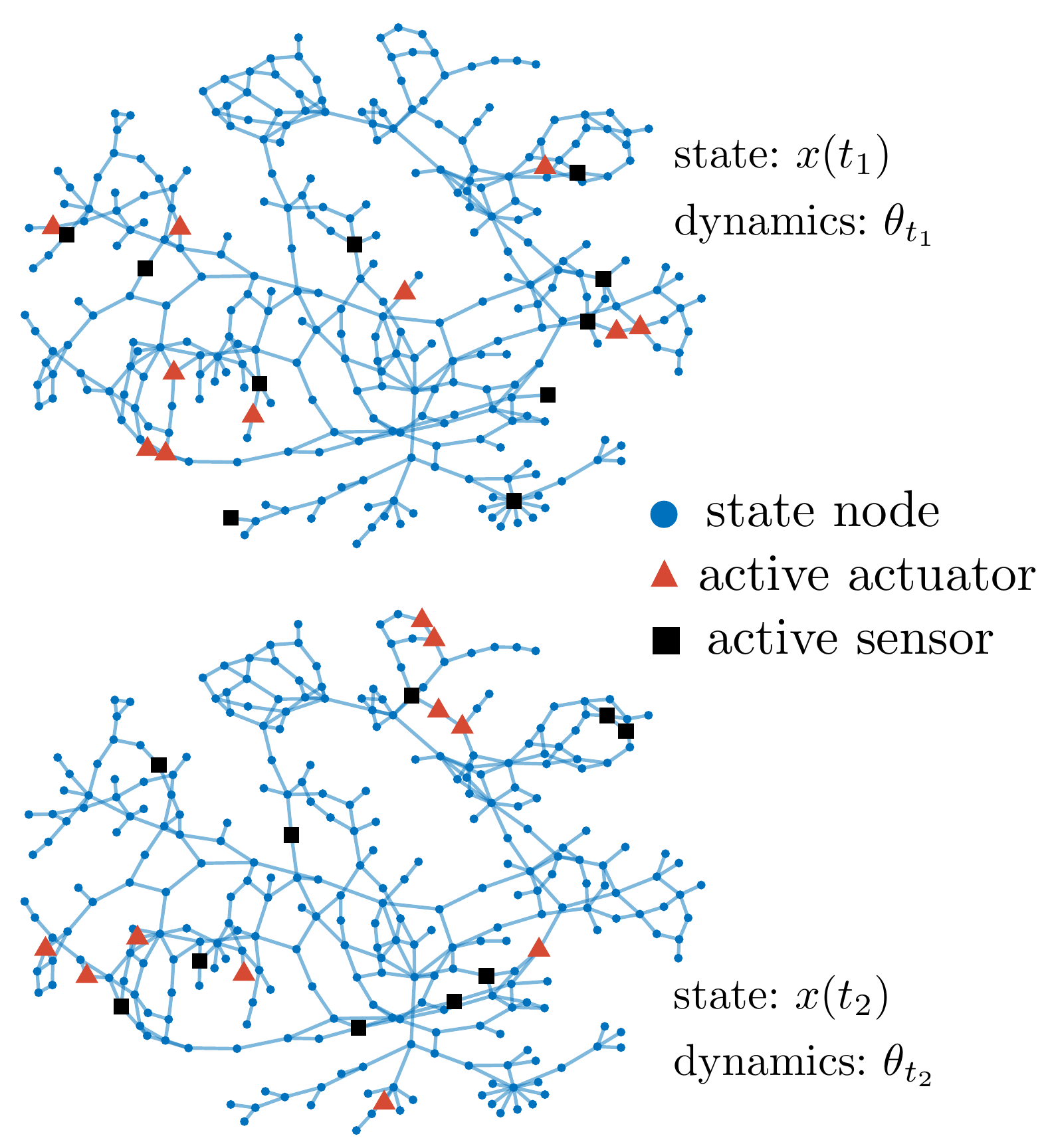}
    \caption{Illustration of a Self-Tuning Network Control Architecture with sensor and actuator locations adapting to the time-varying network state and dynamics.}
    \label{fig:selftuningarch_1}
\end{figure}
We define a finite set of possible actuator locations in the network $\mathbf{U} = \{ u_1, u_2, ..., u_M \}$, each of which corresponds to an input signal $u_i(t) \in \mathcal{U}_i$ that affects the state dynamics in a distinct way. Similarly, we define a finite set of possible sensor locations $\mathbf{Y} = \{ y_1, y_2, ..., y_P \}$, each of which corresponds to a distinct output measurement $y_i(t) \in \mathcal{Y}_i$ related to the network state. For a subsets of actuators $S_t \subset \mathbf{U}$ and sensors $T_t \subset \mathbf{Y}$ that are active at time $t$ and specify input and output signals $u^{S_t}(t) \in \Pi_{i\in S_t} \mathcal{U}_i$ and $y^{T_t}(t) \in \Pi_{i\in T_t} \mathcal{Y}_i$, respectively, the network dynamics are given by
\begin{align}
    x(t+1) &= f_{\theta_t}^{S_t}(x(t), u^{S_t}(t), w(t)) \\
    y^{T_t}(t) &= h_{\theta_t}^{T_t}(x(t), v(t)), 
\end{align}
where $w(t)$ is a stochastic disturbance iid from a distribution function $P_w^{\theta_t}$, $v(t)$ is measurement noise iid from a distribution function $P_v^{\theta_t}$, and $\theta_t$ is a (generally) unknown and time-varying dynamics parameter specifying the dynamics map $f_{\theta_t}^{S_t}$ and measurement map $h_{\theta_t}^{T_t}$. Broadly, the goal is to \emph{jointly} adapt the active sensors and actuators and the control inputs to all available information about the time-varying network state and dynamics parameter to optimize a performance objective. Certain network states and dynamics may render a fixed network control architecture ineffective, which motivates adapting not just the input signal, but also the architecture. When the dynamics parameter is unknown, this may involve a combination of statistical estimation/identification of $\theta_t$ and control design based on an estimate for $\theta_t$ and a representation of the estimation uncertainty.

\subsection{A Cardinality-Constrained Architecture Design Problem} Let $\mathbf{U}_K = \{ S \in 2^\mathbf{U} \mid |S| = K\}$ and $\mathbf{Y}_L = \{ T \in 2^\mathbf{Y} \mid |T| = L\}$ denote all possible actuator and sensor subsets of cardinality $K$ and $L$, respectively. Let $y_{0:t} = [y^{T_0}(0), y^{T_1}(1), ..., y^{T_t}(t)] \in \mathcal{Y}_{0:t}$ and $u_{0:t-1} = [u^{S_0}(0), u^{S_1}(1), ..., u^{S_{t-1}}(t-1)] \in \mathcal{U}_{0:t-1}$ denote the output and input histories, respectively. We define an \emph{architecture-policy} $\pi_t : \mathcal{Y}_{0:t} \times \mathcal{U}_{0:t-1} \rightarrow \mathbf{Y}_L \times \mathbf{U}_K \times \Pi_{i\in S_t} \mathcal{U}_i$ with $S_t \in \mathbf{U}_K$ as a mapping from the input-output history to sensor and actuator subsets and the next input and output, so that $u^{S_t}(t) = \bar \pi_t(y_{0:t}, u_{0:t-1})$ and $y^{T_{t+1}}(t+1) = h_{\theta_{t+1}}^{T_{t+1}}(x(t+1), v(t+1))$. Thus, an architecture-policy specifies \emph{both} a set of active sensors and actuators at each time \emph{and} a feedback policy specifying the control input for active actuators at each time; i.e., $\pi_t$ defines the triple $(S_t, T_{t+1}, \bar \pi_t)$ specifying which actuators to utilize next, which measurements to collect next, and which input values to apply based on all available information. Each component of $(S_t, T_{t+1}, \bar \pi_t)$ depends on the available information. For a finite time horizon of length $\mathcal{T}$, we define $\pi = [\pi_0, \pi_1, ..., \pi_{\mathcal{T}-1}]$.

For a given architecture-policy $\pi$ with $u^{S_t}(t) = \bar \pi_t(y_{0:t,} u_{0:t-1})$, we define a cost function for initial state $x(0) = x$
\begin{equation}
    J_\pi(x) = \mathbf{E}_{w,v} \sum_{t=0}^{\mathcal{T}-1} c_t(x(t), u^{S_t}(t)) + c_\mathcal{T}(x(\mathcal{T})),
\end{equation}
where $c_t : \mathcal{X} \times \Pi_{i\in S_t} \mathcal{U}_i \rightarrow \mathbf{R}$ is a stage cost function, $c_\mathcal{T} : \mathcal{X} \rightarrow \mathbf{R}$ is a terminal cost function, and expectation is taken with respect and disturbance and measurement noise sequences.
The self-tuning network architecture design problem is then to find an optimal architecture-policy, i.e.
\begin{equation} \label{optimalcost}
    J^*(x) = \min_\pi J_\pi(x), \quad\quad  \pi^* \in \arg \min_\pi J_\pi(x).
\end{equation}
This problem is extremely challenging, as it combines already challenging (stochastic, nonlinear, output-feedback, data-driven) feedback control design with a combinatorial architecture search. Generally this of course will require approximation and heuristics for both control and architecture design. Nevertheless we believe this general formulation provides an extremely rich problem space with many exciting possibilities to apply a wide variety of approximation methods from stochastic control, system identification, reinforcement learning, and static network control architecture design.

\begin{figure}[ht]
    \centering
    \includegraphics[width=\linewidth]{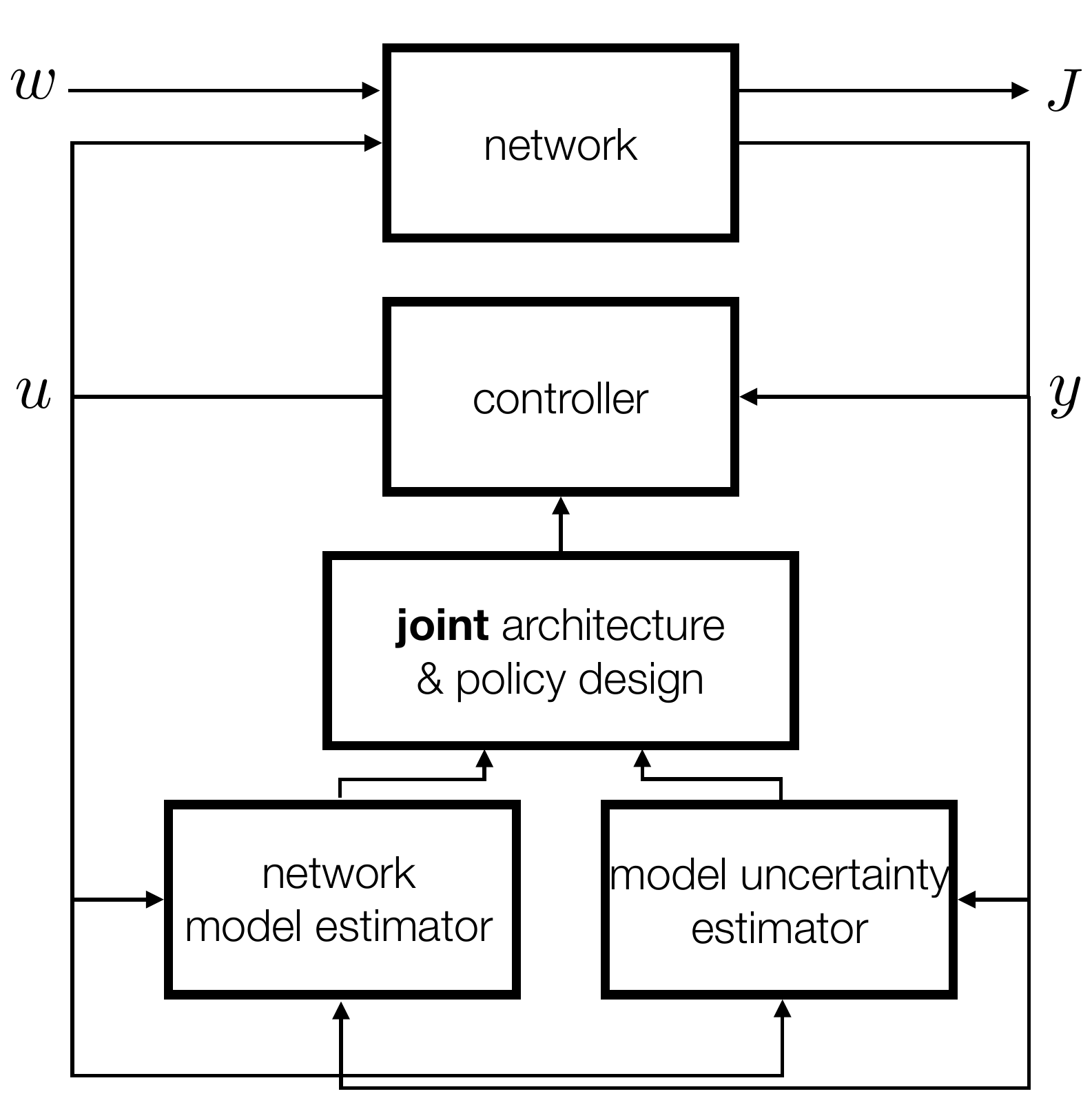}
    \caption{A Self-Tuning Network Control Architecture.}
    \label{fig:selftuningarch_2}
\end{figure}
\subsection{Self-Tuning Architectures} A block diagram illustrating a general self-tuning network control architecture is shown in Figure \ref{fig:selftuningarch_2}, which is analogous to the classical self-tuning regulator from adaptive control. The architecture can be viewed as a pair of coupled loops.  An inner loop features the dynamical network in closed-loop with a feedback controller. An outer loop estimates model parameters $\theta_t$ and a model uncertainty representation (e.g., a probability distribution or an uncertainty set) based on input-output data, which are then fed into a \emph{joint} architecture-controller design module that computes sensor and actuator subsets together with a feedback policy. Our formulation adapts the architecture and policy parameters at the same rate, but it is possible (and potentially computationally advantageous) to use different rates for each of the model estimation, policy parameter adaptation, and architecture adaptation. For example, architecture adaptation may occur at a (much) slower rate than model estimate or policy parameter updates while still providing substantial performance benefits. This approach can also be viewed as maintaining a data-driven ``digital twin'' of the network and adapting both the architecture and policy as the digital twin evolves \cite{tao2018digital}.

The self-tuning architecture structure is highly flexible in terms of the underlying methods for system identification, feedback control design, and network control architecture design. Many combinations have been explored in the literature in the fixed architecture setting. But to our best knowledge, no prior work has considered adaptation of the control architecture.

\section{Dynamic Programming for Self-Tuning Actuators with Full State Feedback}
We consider a special case where the exact network state is available for feedback, and the architecture design problem consists of selecting only actuator subsets at each time based on the current state and dynamics parameter (or an estimate thereof). Let $x_{0:t} = [x(0), x(1), ..., x(t)] \in \mathcal{X}_{0:t}$ denote the network state history. An architecture-policy $\pi_t$ defines an active actuator set and feedback policy $(S_t,\bar \pi_t)$ such that $u^{S_t}(t) = \bar \pi_t(x_{0:t}, u_{0:t-1})$. In general, the self-tuning architecture involves using the available state-input data $(x_{0:t}, u_{0:t-1})$ to estimate the dynamics parameter $\theta_t$, and then selecting an architecture-policy through a control design method. 

When the dynamics parameter $\theta_t$ is known, the Markovian structure allows consideration of architecture-policies of the current state $u^{S_t}(t) = \bar \pi_t(x(t))$. This then permits a dynamic programming solution to the problem stated in the following result.
\begin{thm}
Consider the optimal control problem
\begin{align}
  \begin{array}{ll}
    \min_\pi & \mathbf{E}_{x_0,w} \sum_{t=0}^{\mathcal{T}-1} c_t(x(t), u^{S_t}(t)) + c_\mathcal{T}(x(\mathcal{T}))\\
    \mathrm{subject \ to} & x(t+1) = f_{\theta_t}^{S_t}(x(t), u^{S_t}(t), w(t))\\
  \end{array}\label{eq:3}
\end{align}
The optimal cost function $J^*(x)$ defined in \eqref{optimalcost} is obtained from the last step of the dynamic programming recursion
\begin{align} \label{DP1}
  J_\mathcal{T}(x) &= c_\mathcal{T}(x) \\
  J_{t}(x)  &= \min_{(S_t, u) \in \mathbf{U}_K \times \Pi_{i\in S_t} \mathcal{U}_i } \mathbf{E}_w \left[ c_{t}(x,u^{S_t}) + J_{t+1}(f_{\theta_t}^{S_t}(x, u^{S_t}, w)) \right], \label{DP2}
\end{align}
and corresponding optimal architecture-policies are obtained via
\begin{equation} \label{DPpolicy}
    \pi_t^*(x) \in \underset{(S_t, u) \in \mathbf{U}_K \times \Pi_{i\in S_t} \mathcal{U}_i }{\arg \min} \mathbf{E}_w \left[ c_{t}(x,u^{S_t}) + J_{t+1}(f_{\theta_t}^{S_t}(x, u^{S_t}, w)) \right].
\end{equation}
\end{thm}
\begin{proof}
The proof follows the same inductive argument as in standard dynamic programming \cite{bertsekas2012dynamic}, except that the minimization step is performed jointly over architectures and policies. For each time $t \in [0, \mathcal{T}-1]$, we define the tail architecture-policy $\pi^t = [\pi_t, \pi_{t+1}, ..., \pi_{\mathcal{T}-1}]$, where $\pi_t$ specifies $(S_t, \bar \pi_t)$ with $u^{S_t}(t) = \bar \pi_t(x(t))$. We then define the optimal cost-to-go functions at time $t$ for state $x(t) = x$
\begin{equation*}
    J_t^*(x) = \min_{\pi^t} \mathbf{E} [\sum_{\tau = t}^{\mathcal{T}-1} c_\tau(x(\tau), \bar \pi_\tau(x(\tau))) + c_\mathcal{T}(x(\mathcal{T}))]
\end{equation*}
with expectation with respect to tail disturbance sequences. We define $J^*_{\mathcal{T}}(x) = c_\mathcal{T}(x)$, establishing the inductive base case.

Now assume for some $t$ that $J^*_{t+1}(x) = J_{t+1}(x)$. We have
\begin{align*}
    J_t^*(x) &= \min_{\pi_t} \mathbf{E}_{w} [ c_t(x, \bar \pi_t(x)) + [\min_{\pi^{t+1}} \mathbf{E} [\sum_{\tau = t+1}^{\mathcal{T}-1} c_\tau(x(\tau), \bar \pi_\tau(x(\tau)))] ] \\
    &= \min_{\pi_t} [\mathbf{E}_{w} c_t(x, \bar \pi_t(x)) + J^*_{t+1}(x^+) ] \\
    &= \min_{\pi_t} [\mathbf{E}_{w} c_t(x, \bar \pi_t(x)) + J^*_{t+1}(f_{\theta_t}^{S_t}(x, \bar \pi_t(x), w)) ] \\
    &= \min_{\pi_t} [\mathbf{E}_{w} c_t(x, \bar \pi_t(x)) + J_{t+1}(f_{\theta_t}^{S_t}(x, \bar \pi_t(x), w)) ] \\
    &= \min_{(S_t, u) \in \mathbf{U}_K \times \Pi_{i\in S_t} \mathcal{U}_i } \mathbf{E}_w \left[ c_{t}(x,u^{S_t}) + J_{t+1}(f_{\theta_t}^{S_t}(x, u^{S_t}, w)) \right] \\
    &= J_t(x),
\end{align*}
where the first step follows from the Principle of Optimality and Markovian structure of the dynamics, the second by definition of $J^*_{t+1}$, the third by the state update equation, the fourth by the induction hypothesis, and the fifth by definition of the architecture-policy. This establishes the induction step, and the proof is complete.
\end{proof}

\subsection{Self-Tuning LQR Architectures}
We now consider a further special case  with linear dynamics and convex quadratic costs. The nodal states are real-valued vectors $x_i(t) \in \mathcal{X}_i = \mathbf{R}^{n_i}$, and the network state is then $x(t) \in \mathcal{X} = \mathbf{R}^N$ with $N = \sum_i n_i$. The inputs are also real-valued vectors $u_i(t) \in \mathcal{U}_i = \mathbf{R}^{m_i}$. In this setting the set of possible actuator locations can be identified with a set of columns that can be used to construct an input matrix.
In particular, we can set $\mathbf{U} = \{b_1, b_2, ..., b_M \}$, where $b_i \in \mathbf{R}^N$. For $S_t \subset \mathbf{U}$, we form the input matrix 
\begin{align}
  B^{S_t} = 
  \begin{bmatrix}
    | & | & \cdots & | \\
    b_{i_1} & b_{i_2} & \cdots & b_{i_k}\\
    | & | & \cdots & | 
  \end{bmatrix} \in \mathbf{R}^{N \times |S_t|},
\end{align}
For $|S_t| = K$, there are $\binom{n}{K}$ possible input matrices, up to permutation of the columns. We can identify this set with $\mathbf{U}_K$ so that $\mathbf{U}_K = \{  B^{S_t} \in \mathbf{R}^{N \times |S_t|} \mid S_t \subset \mathbf{U}, |S_t| = K \} $.
The system dynamics with active input matrix $B^{S_t}$ at time $t\in\{0,\dots,\mathcal{T}-1\}$ is given by
\begin{align}
  x(t+1) = Ax(t) + B^{S_t} u^{S_t}(t) + w(t),
\end{align}
where $w(t)$ is an iid $\mathbf{R}^N$-valued zero mean random vector with covariance matrix $W$. The following result establishes a more explicit structure for the optimal cost functions and architecture-policies.

\begin{thm}
Consider the linear quadratic optimal control problem over state-feedback architecture policies $\pi$
\begin{align}
  \begin{array}{ll}
    \min_\pi & \mathbf{E}_{x_0,w}\left[\sum_{t=0}^{\mathcal{T}-1} x_{t}^\tra Q_{t}x_{t} + u_t^\tra  R_t u_t  \right] + x^\tra(\mathcal{T}) Q_\mathcal{T} x(\mathcal{T}) \\
    \mathrm{subject \ to} & x(t+1) = Ax(t) + B^{S_t} u^{S_t}(t) + w(t), \\
         & B^{S_t} \in \mathbf{U}_K.
  \end{array}\label{eq:1}
\end{align}
The optimal cost-to-go functions and policies obtained from the dynamic programming recursion \eqref{DP1}, \eqref{DP2}, \eqref{DPpolicy} for problem~\eqref{eq:1} are piecewise quadratic and piecewise linear, respectively, defined over a finite partition of the state space $\mathbf{R}^N$.
\end{thm}
\begin{proof}
The dynamic programming recursion for \eqref{eq:1} is
\begin{align}
  J_\mathcal{T}(x) &=  x^\tra Q_\mathcal{T} x  \\
  J_{t}(x)  &= \min_{(u,B^{S_t}) \in \mathbb{R}^K\times \mathbf{U}_K} \mathbb{E}\left[ x^\tra Q_{t}x + u^\tra R_{t}u + J_{t+1}(x^+) \right],\label{eq:recursion_dp}
\end{align}
where $x^+ = Ax + B^{S_t} u + w$. To lighten notation we denote the architecture optimization variable $B_t:=B^{S_t}$.
The recursion one timestep backwards from $\mathcal{T}$  can be split into a minimization over $u$ and over $B_{\mathcal{T}-1}$ as,
\begin{align}
  &J_{\mathcal{T}-1}(x)=  \min_{B_{\mathcal{T}-1}\in\mathbf{U}_K}~\min_{u \in \mathbb{R}^K}
  \mathbb{E}\left[ x^\top Q_{\mathcal{T}-1}x + u^\top R_{\mathcal{T}-1}u\right. \label{eq:jtmin1} \\
  & \left.+ (Ax + B_{\mathcal{T}-1} u + w)^\top Q_{\mathcal{T}}(Ax + B_{\mathcal{T}-1} u + w) \right].\nonumber
\end{align}
Evidently, the inner minimization over $u$ yields,
\begin{align}
 u_{\mathcal{T}-1}^*(x) = -\left(B^{S_{\mathcal{T}-1} \top} Q_{\mathcal{T}} B_{\mathcal{T}-1} + R\right)^{-1}B^{S_{\mathcal{T}-1}\top} Q_{\mathcal{T}}A x    .\label{eq:opt_u}
\end{align}
Substituting \eqref{eq:opt_u} into \eqref{eq:jtmin1} yields that, for a given $B_{\mathcal{T}-1}$, the function inside the minimum over $B_{\mathcal{T}-1}$ is quadratic.
Since there are a finite number of choices of $B_{\mathcal{T}-1}$, each depending on $x$, it follows that $J_{\mathcal{T}-1}(x)$ is piecewise quadratic, and $u_{\mathcal{T}-1}^*(x)$ piecewise linear.

The remainder of the proof can be done by induction on a standard ansatz, as follows.
Suppose $J_t(x) = x^\top P_t(B_t(x)) x + q_t(B_t(x))$, where $P_t(B_t(x))=:P_t^{B_t}$ and $q_t(B_t(x))$ depend on the choice of $B_t$ at time $t$, which in turn depends on $x$; in other words that $J_t(x)$ is piecewise quadratic.
We can write,
\begin{align}
  \begin{split}
    &J_{t-1}(x)   = \min_{B_{t-1}\in\mathbf{U}_K}~\min_{u \in \mathbb{R}^K}\Bigg[
    \mathbf{tr}(P_t^\Bt W) + q_t(B_t(x^+))+\\
\quad &
      \begin{bmatrix}
        x \\ u
      \end{bmatrix}^\top 
      \begin{bmatrix}
        A^\top  P_t^\Bt  A + Q_{t-1} & A^\top  P_t^\Bt  B_{t-1} \\
        B_{t-1}^\top P_t^\Bt  A & B_{t-1}^\top P_t^\Bt B_{t-1} + R_{t-1}
      \end{bmatrix}
      \begin{bmatrix}
        x \\ u
      \end{bmatrix}\Bigg]
  \end{split}\\
\begin{split}
      &=\min_{B_{t-1}\in\mathbf{U}_K} q_t(B_t(x^+))+\mathbf{tr}(P_t^\Bt W) +
     x^\top \left[ A^\top P_t^\Bt A + Q_{t-1} \right. \\  & - A^\top P_t^\Bt B_{t-1}\left( B_{t-1}P_t^\Bt B_{t-1} + R_{t-1} \right)^{-1}B_{t-1}^\top P_t^\Bt A \Big] x 
\end{split}\\
    &=\min_{B_{t-1}\in\mathbf{U}_K} \left[ x^\top P_{t-1}^\Btm x \right] + q_{t-1}(B_{t-1}(x)),\label{eq:2}
\end{align}
where $\mathbb{E}[w_t^TP_{t+1}(B_t(x))w_t] := \mathbf{tr}(P_t^{B_t}W)$, $\mathbf{tr}(P_t^\Bt W) + q_t(B_t(x^+))$ $:=q_{t-1}(B_{t-1}(x))$, and
where the problem in the last display can be solved by exhaustive search over the $\binom{n}{K}$ elements of $\mathbf{U}_K$.
It follows that \eqref{eq:2} is piecewise-continuous in $x$, and piecewise quadratic.
\end{proof}

\begin{figure}
    \centering
    \includegraphics[width=\columnwidth]{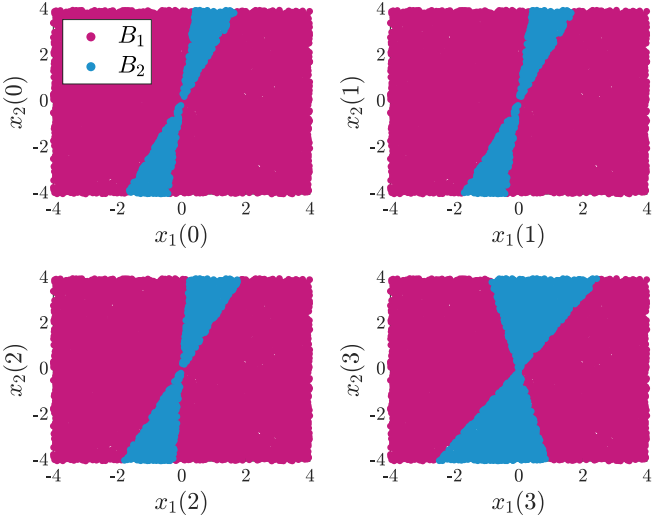}
    \caption{10000 points $x(t)\in\mathbb{R}^2$ of~\eqref{eq:eg} sampled in $[-4,4]^2$ over $\mathcal{T} = [0,\dots,4]$ with the optimal actuator highlighted at times $[0,\dots,3]$. Magenta denotes actuator 1, blue actuator 2.}
    \label{fig:partition}
\end{figure}

We show a numerical example where we highlight the `partition' of $\mathbf{R}^N$ formed by the piecewise continuity of the value function.
Consider the system,
\begin{align}
    A = 
    \begin{bmatrix}
       -2.2639&    0.6379\\
       -0.2619  &  0.6383
    \end{bmatrix},~B \in \left\{
        \begin{bmatrix}
        1\\0
        \end{bmatrix},~
        \begin{bmatrix}
        0\\1
        \end{bmatrix}\right\},\label{eq:eg}
\end{align}
with cost matrices $Q = \mathrm{blkdiag}\{1,2\}$ and $R=1$, and a time horizon of $|\mathcal{T}|=5$.
By exhaustive search over all trajectories of length $5$, we find the optimal actuator assignment given a starting point $x(0)\in[-4,4]^2$.
The optimal actuators given a point $x(t)$ for $t\in [0,\dots,3]$ are highlighted in Figure~\ref{fig:partition}.

This result provides an exact algorithm to construct the optimal cost-to-go functions and policies. However, the exhaustive search over architectures is computationally intractable for all but the smallest actuator set cardinalities. This motivates approximation algorithms for joint architecture-policy optimization. 

\subsection{A Greedy Heuristic for Self-Tuning LQR Architectures}
We now present a simple greedy heuristic to approximate an architecture-policy. For concreteness, we focus on self-tuning LQR architectures with a system identification component for a static dynamics matrix. The network dynamics are
\begin{equation}
    x(t+1) = A_{\theta} x(t) + B^{S_t} u^{S_t}(t) + w(t),
\end{equation}
where $w(t)$ is iid zero mean with covariance $W$.
Based on the state-input history $(x_{0:t}, u_{0:t-1})$ the dynamics parameter $\theta$ is identified using least-squares estimation, and then a greedy algorithm is used to select a subset of actuators of cardinality $K$ to optimize the infinite-horizon LQR control performance associated with the network model estimated. In particular, the approximate architecture-policy is
\begin{equation}
    u^{S_t^{greedy}}(t) = K^{S_t^{greedy}} x(t),
\end{equation}
where the computation of $K^{S_t^{greedy}}$ is described in Algorithm \ref{algo:greedy-arch-policy}. Although this greedy algorithm is suboptimal, it renders the joint architecture-policy design computationally tractable and, as we will see next, can still provide dramatically improved performance compared to fixed architectures.


\begin{algorithm}
\caption{Greedy Self-Tuning Architecture-Policy Approximation}
\begin{algorithmic}[1]
\label{algo:greedy-arch-policy}
    \REQUIRE State-input history $(x_{0:t}, u_{0:t-1})$ at time $t$, actuator location set $\mathbf{U} = \{b_1, b_2, ..., b_M \}$, cost matrices $Q$, $R$, actuator set cardinality $K$
    \STATE Identify dynamic parameter:\\ $\hat \theta = \arg \min_{\theta} \sum_{\tau=0}^{t-1} \| x(\tau+1) - (A_{\theta} x(\tau) + B^{S_\tau} u^{S_\tau}(\tau)) \|^2$.
    \STATE Initialize: $S_t = \emptyset$, $B^{S_t} = []$.
    \WHILE{$|S_t| < K$}
    \STATE $s^* = \arg \min_{s \in \mathbf{U}} x(t)^T P_s x(t)$ where $P_s = Q + A_{\hat \theta}^T P_s A_{\hat \theta} - A_{\hat \theta}^T P_s B_s (R + B_s^\top P_s B_s)^{-1} B_s^\top P_s A_{\hat \theta}$ and $B_s = [B^{S_1} \ b_s]$
    \STATE $S_t \leftarrow S_t \cup \{ s^* \}$, $B^{S_t} = [B^{S_t} \ b_{s^*}]$
    \ENDWHILE
    \STATE $P = Q + A_{\hat \theta}^T P A_{\hat \theta} - A_{\hat \theta}^T P B^{S_t } (R + B^{S_t \top} P B^{S_t})^{-1} B^{S_t \top} P A_{\hat \theta}$
    \STATE $K^{S_t^{greedy}} = -(R + B^{S_t \top} P B^{S_t})^{-1} B^{S_t} P A_{\hat \theta}$
\ENSURE $u^{S_t^{greedy}}(t) = K^{S_t^{greedy}} x(t)$
\end{algorithmic}
\end{algorithm}







\section{Numerical Experiments} \label{sec:equivalence}
\subsection{A Simple Example}
We begin with a simple example that clearly demonstrates the potential benefits of a self-tuning network control architecture over a fixed architecture. Consider the system
\begin{equation}
    x(t+1) = A_{\theta_t} x(t) + B^{S_t} u(t) + w(t),
\end{equation}
where the dynamics matrix switches occasionally (the exact nature of the switching is not needed to qualitatively illustrate the basic benefits) between two possible values $A_{\theta_t} \in \left\{\left[\begin{array}{cc}1 & 0.5 \\0.5 & 1\end{array}\right], \left[\begin{array}{cc}1 & -0.5 \\-0.5 & 1\end{array}\right]  \right\}$, the control architectures consist of two possible actuator locations $\mathbf{U} = \{b_1 = [1, 1]^T, b_2 = [1, -1]^T \}$, and the disturbance is zero-mean Gaussian $w(t) \sim \mathcal{N}(0, \sigma^2 I)$. We compare a fixed architecture with $B^{S_t} = b_1 \ \forall t$ and a self-tuning architecture that allows switching between the two actuators (with only 1 active at each time) based on the current state and (an estimate of) which dynamics matrix is active.

A close examination reveals that the pair $(A_2, b_1)$ is unstabilizable, so under the fixed architecture the state grows without bound when $A_2$ is active. It is not difficult to define a switching signal where $A_2$ is active sufficiently often and the closed-loop is unstable, resulting in infinite cost. With such a fixed architecture, no feedback controller, adaptive or otherwise, can possibly stabilize the network dynamics. On the other hand, the pair $(A_2, b_2)$ is stabilizable. If it is known (or can be identified sufficiently fast) which dynamic matrix is active during certain time intervals, obviously it is highly advantageous to activate actuator $b_2$ whenever dynamics matrix $A_2$ is active. With a self-tuning architecture, the network can be easily stabilized with near optimal cost, even when only one actuator is activated at each time step.

This simple example demonstrates the significant potential value of self-tuning architectures. More generally, self-tuning architectures will be valuable whenever the network state and/or dynamics change in ways that render an existing set of sensors and actuators ineffective, and alternative, more effective sets are available to activate. The challenges lie in detecting or estimating such changes from data as the network evolves and employing effective heuristics for joint architecture and feedback control design.


\subsection{A Self-Tuning LQR Example}
In this section, we describe a larger self-tuning LQR example. A simple greedy actuator selection heuristic for joint architecture and control design enables a self-tuning architecture to provide dramatically improved performance over a fixed architecture. This improvement is realized even with known linear time-invariant dynamics, with the actuator set selected at each time based on the current state.

We consider 50-node network with a scalar state for each node and a randomly generated unstable dynamics matrix. The set of possible actuators consists of the standard unit basis vectors $\mathbf{U} = \{ e_{1},...,e_{25} \}$, so each actuator can inject an input signal at each node. The cost matrices are $Q_t = I_n$ and $R_t = I_{|S_t|}$, i.e., the input cost matrix is identical for every set of actuators. The disturbance is iid with $w(t) \sim \mathcal{N}(0, 1e-4)$. We randomly generated in initial state $x_0 \sim \mathcal{N}(0, 25)$. The number of actuators available at each time is limited to $K=2$. We simulated the network dynamics with a fixed architecture $B = [e_1 \ e_2]$ using the optimal LQR policy, and we simulated with a self-tuning architecture using a greedy heuristic as described above. Figure \ref{fig:cost} shows a comparison of typical costs for each architecture, and Figure \ref{fig:states} shows corresponding state trajectories. The cost of the fixed architecture is far worse, a factor 80x more, than the cost of the self-tuning architecture. Code and problem data for implementing the algorithm can be found at \cite{gitlabcode}.
\begin{figure}[ht]
    \centering
    \includegraphics[width=\linewidth]{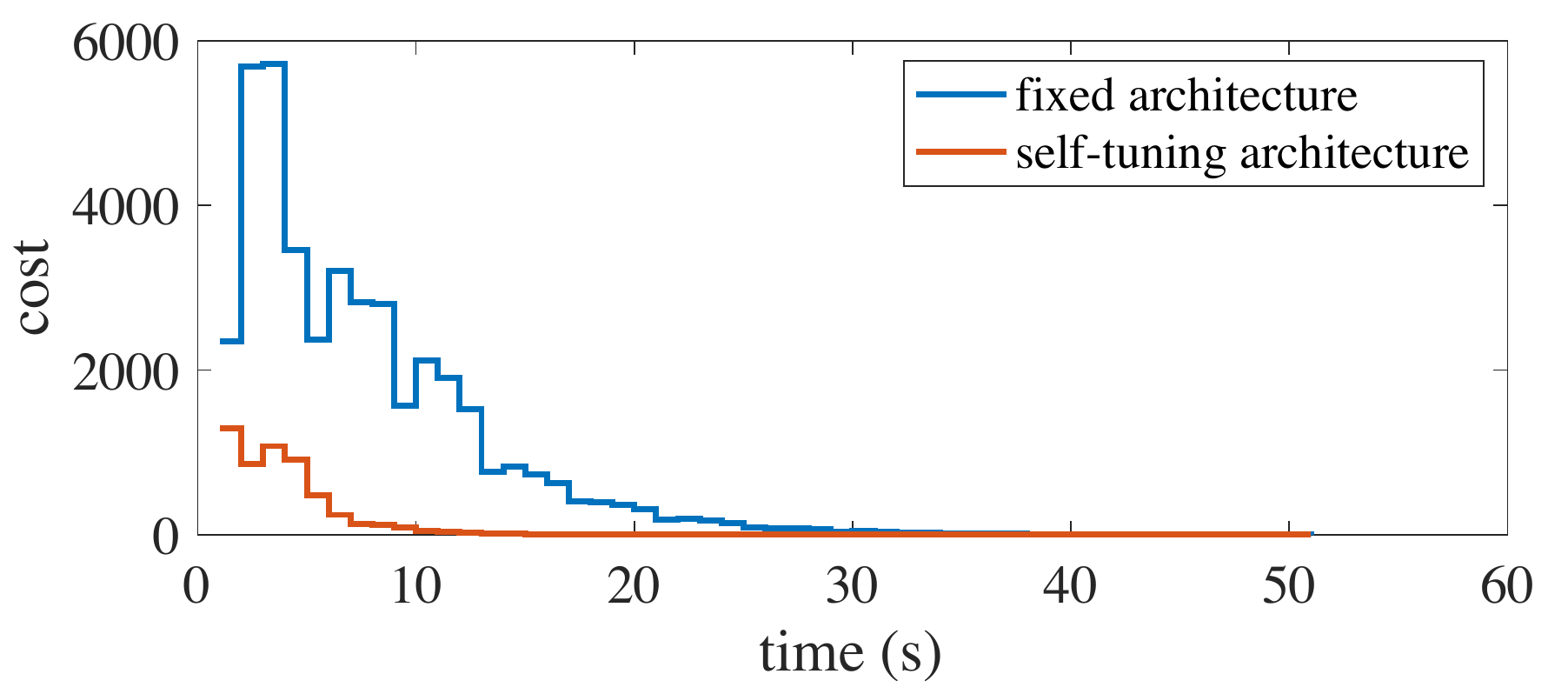}
    \includegraphics[width=\linewidth]{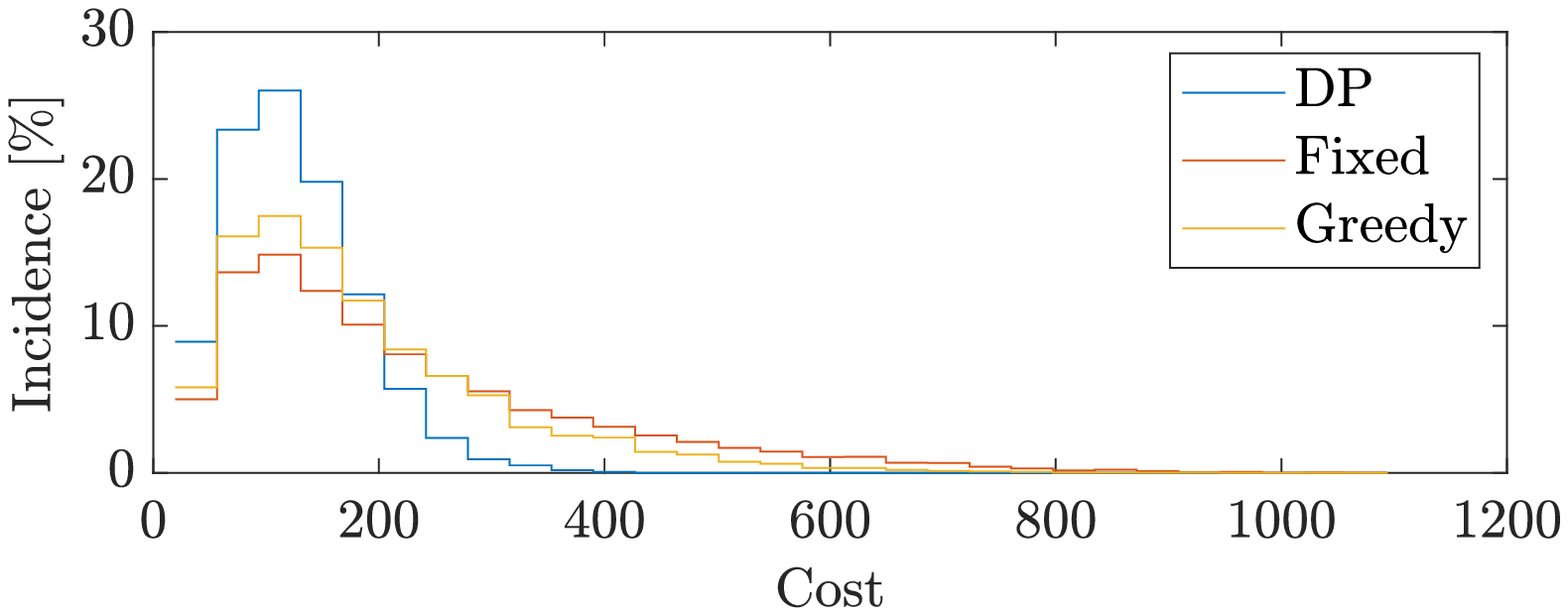}
    \caption{Cost comparison, fixed vs. self-tuning architecture.}
    \label{fig:cost}
\end{figure}

\begin{figure}[ht]
    \centering
    \includegraphics[width=\linewidth]{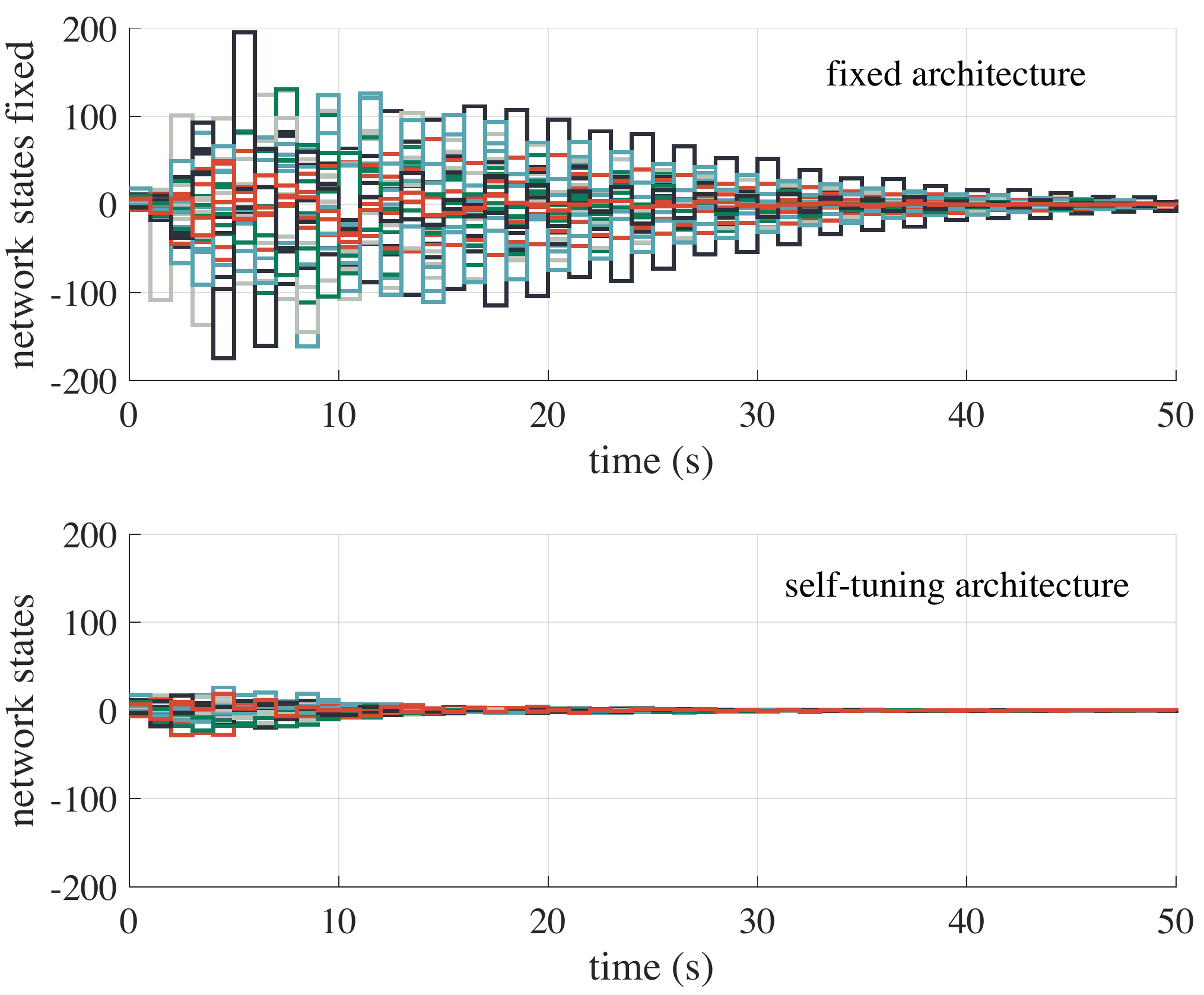}
    \caption{State trajectories, fixed vs. self-tuning architecture.}
    \label{fig:states}
\end{figure}


\section{Conclusions and future work} \label{sec:conclusion}
We formulated a general mathematical framework for self-tuning network control architecture design and proposed a general solution structure analogous to the classical self-tuning regulator from adaptive control.

We have barely scratched the surface and believe there are many opportunities to apply a wide variety of approximation methods from stochastic control, system identification, reinforcement learning, and static architecture design.


\bibliographystyle{IEEEtran}
\bibliography{bibliography.bib}
\end{document}